\documentclass[12pt]{amsart}
\usepackage{amsmath,amssymb,amsthm}
\usepackage{latexsym,ulem}      
\usepackage{color,colordvi,graphicx}
\numberwithin{equation}{section}
\newtheorem{theorem}{Theorem}
\newtheorem{proposition}[theorem]{Proposition}
\newtheorem{lemma}[theorem]{Lemma}
\newtheorem{corollary}[theorem]{Corollary}

\theoremstyle{remark}

\newtheorem{remark}[theorem]{Remark}

\newcommand{\defcolor}[1]{\Blue{#1}}
\newcommand{\demph}[1]{\defcolor{{\sl #1}}}

\newcommand{\C}{{\mathbb C}}
\newcommand{\R}{{\mathbb R}}

\newcommand{\calL}{{\mathcal L}}
\newcommand{\calO}{{\mathcal O}}
\newcommand{\calU}{{\mathcal U}}
\newcommand{\calX}{{\mathcal X}}

\newcommand{\calZ}{{\mathcal Z}}

\newcommand{\Edot}{E_{\bullet}}
\newcommand{\Fdot}{F_{\bullet}}
\newcommand{\blambda}{\boldsymbol{\lambda}}

\DeclareMathOperator{\codim}{codim}
\DeclareMathOperator{\Sym}{Sym}

\newcommand{\Sp}{\mbox{\rm Sp}}
\newcommand{\Fl}{{\mathbb F}\ell}
\newcommand{\pr}{\mbox{\rm pr}}
\newcommand{\Gr}{\mbox{\rm Gr}}
\newcommand{\LG}{\mbox{\rm LG}}

\title[A congruence modulo four for real Schubert calculus]{A congruence modulo four for real\\
      Schubert calculus with isotropic flags}

\author{Nickolas Hein}
\address{Nickolas Hein \\
         Department of Mathematics\\
         University of Nebraska at Kearney\\
         Kearney\\
         Nebraska \ 68849\\
         USA}
\email{heinnj@unk.edu}
\urladdr{http://www.unk.edu/academics/math/faculty/About\_Nickolas\_Hein/}
\author{Frank Sottile}
\address{Frank Sottile \\
         Department of Mathematics\\
         Texas A\&M University\\
         College Station\\
         Texas \ 77843\\
         USA}
\email{sottile@math.tamu.edu}
\urladdr{http://www.math.tamu.edu/~sottile}

\author{Igor Zelenko}
\address{Igor Zelenko\\
         Department of Mathematics\\
         Texas A\&M University\\
         College Station\\
         Texas \ 77843\\
         USA}
\email{zelenko@math.tamu.edu}
\urladdr{http://www.math.tamu.edu/~zelenko}
\thanks{Research of Sottile and Hein supported in part by NSF grant DMS-1001615}
\keywords{Lagrangian Grassmannian, Wronski map, Shapiro Conjecture}
\subjclass[2010]{14N15, 14P99}

\begin{document}

\begin{abstract}
 We previously obtained a congruence modulo four for the number of real solutions to many Schubert problems on
 a square Grassmannian given by osculating flags.
 Here, we consider Schubert problems given by more general isotropic flags, and prove this 
 congruence modulo four for the largest class of Schubert problems that could be expected to exhibit this
 congruence.
\end{abstract}

\maketitle

%
\section*{Introduction}
 The number of real solutions to a system of real equations is congruent to the number of
 complex solutions modulo two.
 In~\cite{HSZ}, we established a congruence modulo four for many symmetric
 Schubert problems given by osculating flags, leaving as a conjecture a stronger form of
 that result.
 We prove this conjecture for symmetric Schubert problems given by flags that are
 isotropic with respect to a symplectic form, giving a simpler proof of a stronger 
and more basic result than that obtained in~\cite{HSZ}.

This congruence modulo four follows from a result on the real points in
fibers of a map between real varieties equipped with an involution.
When the fixed point set of the involution has codimension at least two, the number
of real points satisfies a congruence modulo four.
There is an involution acting on symmetric Schubert problems given by
isotropic flags and we can compute the dimension
of the fixed point locus in a universal family of Schubert problems.
Our inability to compute this dimension when the flags are osculating was the obstruction
to establishing the conjecture in~\cite{HSZ}.

The congruence modulo four often implies a non-trivial lower bound on the number of
real solutions to a symmetric Schubert problem given by isotropic flags.
Similar lower bounds and congruences in real algebraic geometry have been of significant
interest~\cite{AH,EG01,FK,Lower,IKS03,IKS04,OT1,OT2,SS,W}.
Another topological study was recently made of this phenomena in the Schubert
calculus~\cite{FM}, and delicate lower bounds~\cite{MT} were given by computing the
signature of a hermitian matrix arising in the proof of the Shapiro
Conjecture~\cite{MTV}.  

In Section~\ref{S:one} we state our main result, whose proof 
occupies Section~\ref{S:two}.

\section{Symmetric Schubert Problems}\label{S:one}

Let $V$ be a complex vector space of dimension $2m$ equipped with a nondegenerate alternating
form $\langle\ ,\ \rangle\colon V\otimes V\to \C$.
Write $\overline{W}$ for the complex conjugate of a point, vector, subspace, or variety $W$.
A variety $W$ is \demph{real} if it is defined by real equations; equivalently, if
$\overline{W}=W$.
Write \defcolor{$W(\R)$} for the real points of a real variety $W$, those that are fixed
by complex conjugation.
Write \defcolor{$S_a$} for the symmetric group of permutations of $\{1,\dotsc,a\}$.

The set of $m$-dimensional linear subspaces of $V$ forms the Grassmannian, $\Gr(m,V)$,
which is a manifold of dimension $m^2$.
A \demph{flag} is a sequence 
$\Fdot\colon F_1\subsetneq F_2\subsetneq \dotsb\subsetneq F_{2m}=V$ of linear subspaces of
$V$ with $\dim F_i=i$.
A \demph{partition} is a weakly decreasing sequence of integers 
$\lambda\colon m\geq\lambda_1\geq\dotsb\geq\lambda_m\geq 0$.
A flag $\Fdot$ and a partition $\lambda$ determine a Schubert subvariety of $\Gr(m,V)$,
\[
  \defcolor{X_\lambda\Fdot}\ :=\ 
   \{H\in\Gr(m,V)\;\mid\; \dim H\cap F_{m+i-\lambda_i}\geq i\ \mbox{ for }i=1,\dotsc,m\}\,.
\]
This has codimension $\defcolor{|\lambda|}:=\lambda_1+\dotsb+\lambda_m$ in $\Gr(m,V)$.

Let  $\blambda=(\lambda^1,\dotsc,\lambda^s)$ be a list of partitions and 
 $\Fdot^1,\dotsc,\Fdot^s$ be  general flags.
By Kleiman's Transversality Theorem~\cite{KL74}
the intersection
 \begin{equation}\label{Eq:SchubProblem}
  X_{\lambda^1}\Fdot^1\,\cap\,
  X_{\lambda^2}\Fdot^2\,\cap\,\dotsb\,\cap\,
  X_{\lambda^s}\Fdot^s\,.
 \end{equation}
is either empty or has dimension $\dim\Gr(m,V)-|\lambda^1|-\dotsb-|\lambda^s|$.
Call $\blambda$ a \demph{Schubert problem} if this expected dimension is zero so
that~\eqref{Eq:SchubProblem} is either empty or consists of finitely many points. 
The number of points \demph{$d(\blambda)$} in~\eqref{Eq:SchubProblem} is independent of the choice
of general flags. 
We will assume that $d(\blambda)\neq 0$.
A choice of flags is an \demph{instance} of the Schubert problem $\blambda$;
its \demph{solutions} are the points in~\eqref{Eq:SchubProblem}.
The instance is \demph{real} if for all $i$, there is some $j$ with
$\overline{\Fdot^i}=\Fdot^j$ and $\lambda^i=\lambda^j$, for then~\eqref{Eq:SchubProblem}
is stable under complex conjugation. 

\begin{remark}\label{R:osculating}
 Osculating flags provide a rich source of isotropic flags.
 As explained in Section~3 of~\cite{HSZ}, a rational normal curve
 $\gamma\colon\C\to V$ 
 induces a symplectic form on $V$ and a symplectic form on $V$ gives rise to a
 rational normal curve, and we may assume that $\gamma$ is real in that
 $\gamma(\overline{t})=\overline{\gamma(t)}$.
 If $\gamma$ is a rational normal curve corresponding to the symplectic form
 $\langle\,,\,\rangle$, then every osculating flag is isotropic.
 (For $t\in\C$, the \demph{osculating flag $\Fdot(t)$} is the flag whose $i$-plane $F_i(t)$ is
 spanned by $\gamma(t)$ and its derivatives $\gamma'(t),\dotsc,\gamma^{(i-1)}(t)$.)

 The study of real solutions to Schubert problems given by flags osculating at real points in Grassmannians
 and flag manifolds has been quite rich and fruitful~\cite{EG02,MSJ,MTV,MTV_R,Purbhoo,So99}.
\end{remark}

A partition $\lambda$ is represented by its Young diagram, which is a left-justified array
of boxes with $\lambda_i$ boxes in row $i$.
We display some partitions with their Young diagrams,
\[
   (2,1,1)\ \longleftrightarrow\ \raisebox{-3.5pt}{\includegraphics{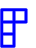}}\,,
   \qquad
   (2,2)\ \longleftrightarrow\ \raisebox{-1.75pt}{\includegraphics{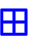}}\,,
   \qquad\mbox{and}\qquad
   (3,2,1)\ \longleftrightarrow\ \raisebox{-3.5pt}{\includegraphics{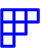}}\,.
\]
A partition $\lambda$ is \demph{symmetric} if it is symmetric about its main diagonal,
that is, if $\lambda=\lambda'$, where \defcolor{$\lambda'$} is the transpose of $\lambda$.
The partitions $(2,2)$ and $(3,2,1)$ are symmetric while $(2,1,1)$ is not.
A Schubert problem $\blambda$ is \demph{symmetric} if every partition in $\blambda$ is symmetric.

Recall that our vector space $V$ was equipped with a nondegenerate alternating bilinear
form $\langle\,,\,\rangle$.
A linear subspace $W$ of $V$ has annihilator $\angle(W)$ under $\langle\,,\,\rangle$,
\[
    \defcolor{\angle(W)}\ :=\ 
     \{v\in V\,\mid\, \langle v,w\rangle=0\ \mbox{\rm for all }w\in W\}\,,
\]
and we have $\dim W + \dim \angle(W)=2m$.
This induces a map $H\mapsto \angle(H)$ on $\Gr(m,V)$ called the \demph{Lagrangian involution}.
Given a flag $\Fdot$, we get the flag \demph{$\angle(\Fdot)$} whose $i$-plane is $\angle(F_{2m-i})$.
A flag $\Fdot$ is \demph{isotropic} if $\angle(\Fdot)=\Fdot$.

The \demph{length, $\ell(\lambda)$} of a symmetric partition is the number of boxes on its
main diagonal, so $\ell(2,2)=2$ while $\ell(2,1)=1$.
We state our main theorem.

\begin{theorem}\label{Th:main}
 Suppose that $\blambda=(\lambda^1,\dotsc,\lambda^s)$ is a symmetric Schubert problem on
 $\Gr(m,V)$ and that $\Fdot^1,\dotsc,\Fdot^s$ are isotropic flags defining a
 real instance of the Schubert problem $\blambda$ such that~$\eqref{Eq:SchubProblem}$ is
 finite. 
 If\/ $\sum_i\ell(\lambda^i)\geq m{+}4$, then the
 number (counted with multiplicity) of real points in~$\eqref{Eq:SchubProblem}$ is
 congruent to the number $d(\blambda)$ of complex points, modulo four. 
\end{theorem}

\begin{remark}
 We show in Remark~\ref{R:last} that $\sum_i\ell(\lambda^i)\geq m$ and this sum has the
 same parity as $m$, so that the condition in Theorem~\ref{Th:main} for this congruence
 modulo four is that $\sum_i\ell(\lambda^i)$ is not equal to $m$ or to $m{+}2$, which is
 very mild.
\end{remark}

 We use the observation of Remark~\ref{R:osculating} that osculating flags are isotropic to deduce a
corollary about Schubert problems given by osculating flags. 
Fix a real rational normal curve $\gamma\colon\C\to V$ with corresponding osculating flags $\Fdot(t)$ for
$t\in \C$, and a symmetric Schubert problem $\blambda=(\lambda^1,\dotsc,\lambda^s)$ on
$\Gr(m,V)$.
An \demph{osculating instance} of this Schubert problem is a list of distinct complex numbers $t_1,\dotsc,t_s$
which give corresponding osculating flags $\Fdot(t_1),\dotsc,\Fdot(t_s)$.
This osculating instance is \demph{real} if $\overline{t_i}=t_j$ implies that $\lambda_i=\lambda_j$.

 We deduce a corollary to Theorem~\ref{Th:main} that implies Conjecture~21 of~\cite{HSZ}, which was the 
strongest result one could reasonably expect to hold concerning this congruence modulo four for symmetric
Schubert problems.
This is strictly stronger than all congruence results obtained in~\cite{HSZ}.

\begin{corollary}
 Suppose that $\blambda=(\lambda^1,\dotsc,\lambda^s)$ is a symmetric Schubert problem on
 $\Gr(m,V)$ and that $\Fdot(t_1),\dotsc,\Fdot(t_s)$ are osculating flags defining a
 real instance of $\blambda$.
 If $\sum_i\ell(\lambda^i)\geq m{+}4$, then the
 number (counted with multiplicity) of real points in
\[
  X_{\lambda^1}\Fdot(t_1)\,\cap\,
  X_{\lambda^2}\Fdot(t_2)\,\cap\,\dotsb\,\cap\,
   X_{\lambda^s}\Fdot(t_s)\,.
\]
 is congruent to the number $d(\blambda)$ of complex points, modulo four.
\end{corollary}

We do not need to assert that the intersection consists of finitely many points, for it always
  does~\cite{EH83}.

\begin{remark}
 When  $d(\blambda)$ is congruent to two modulo four and 
 $\sum_i\ell(\lambda^i)\geq m{+}4$, there will always be at least two real solutions to a
 real instance of a symmetric Schubert problem.
 Such lower bounds implied by Theorem~\ref{Th:main} occur frequently.
 Table~\ref{Ta:one} gives the total number of symmetric Schubert problems in $\Gr(m,V)$
 for small values of $m$, together with the number of those for which Theorem~\ref{Th:main} implies
 a lower bound of two.
 The timings are reported in GHz-seconds (s), GHz-hours (h), and GHz-years (y).
 \begin{table}[htb]
  \caption{Numbers of symmetric Schubert problem with a lower bound of two.}\label{Ta:one} 
  \begin{tabular}{|r||r|r|r|r|r|r|}\hline
   $m$ & 2&3&4&5&6&7\\\hline\hline
   Symmetric&1&8&81& 1037  & 16933&349844\\\hline
   Have lower bound& 0&2&14&199&3289&82753\\\hline
   Percentage& 0&25&17.3&19.2&19.4&23.7\\\hline
   Time &0.57 s&0.6 s&1.9 s&158 s&17.1 h&1.15 y\\\hline
  \end{tabular}
 \end{table}
\end{remark}

\section{Proof of Theorem~\ref{Th:main}}\label{S:two}

We follow the main line of argument for the results of~\cite{HSZ}.
We observe that the Lagrangian involution $H\mapsto \angle(H)$ permutes the solutions to an
instance of a symmetric Schubert problem $\blambda$ given by isotropic flags and then
construct a family $\calX_{\blambda}\to\calZ_{\blambda}$ whose base parameterizes instances of
the Schubert problem $\blambda$ given by isotropic flags and whose fibers are the
solutions to those instances.
We then estimate the codimension of the $\angle$-fixed point locus of the family
$\calX_{\blambda}\to\calZ_{\blambda}$, which shows that the numerical condition 
$\sum_i\ell(\lambda^i)\geq m{+}4$ implies that the fixed points have codimension at
least two.
Finally, we invoke a key lemma from~\cite{HSZ} to complete the proof.

\subsection{The Lagrangian Grassmannian}
An $m$-dimensional subspace $H$ of $V$ is \demph{Lagrang\-ian} if $\angle(H)=H$.
The set of all Lagrangian subspaces of $V$ forms the \demph{Lagrangian Grassmannian}
\defcolor{$\LG(V)$}.
This is smooth of dimension $\binom{m{+}1}{2}$ and is a homogeneous space for the
symplectic group \defcolor{$\Sp(V)$} of linear transformations of $V$ which preserve
$\langle\,,\,\rangle$. 

An isotropic flag $\Fdot$ and a symmetric partition $\lambda$ determine a Schubert
subvariety $Y_\lambda\Fdot$ of $\LG(V)$, which is the intersection
$X_\lambda\Fdot\cap\LG(V)$,
\[
   \defcolor{Y_\lambda\Fdot}\ :=\ 
   \{H\in\LG(V)\;\mid\; \dim H\cap F_{m+i-\lambda_i}\geq i\ \mbox{ for }i=1,\dotsc,m\}\,.
\]
This has codimension $\defcolor{\|\lambda\|}:=\frac{1}{2}(|\lambda|+\ell(\lambda))$ in
$\LG(V)$. 

We need the following result which partially explains why these Lagrangian Schubert
varieties are relevant for Theorem~\ref{Th:main}.

\begin{proposition}[Lemma~9 of~\cite{HSZ}]
 Let $\lambda$ be a partition and $\Fdot$ a flag.
 Then
\[
    \angle(X_\lambda\Fdot)\ =\ X_{\lambda'}\angle(\Fdot)\,.
\]
\end{proposition}

Thus if $\lambda$ is symmetric and $\Fdot$ isotropic, then 
$\angle(X_\lambda\Fdot)=X_\lambda\Fdot$ and $Y_\lambda=(X_\lambda\Fdot)^\angle$, the set
of points of $X_\lambda\Fdot$ that are fixed by $\angle$.
This has the following consequence.

\begin{corollary}
 The Lagrangian involution permutes the solutions to a symmetric Schubert problem given by
 isotropic flags.
\end{corollary}

\subsection{Families associated to Schubert problems}

Let $\blambda$ be a symmetric Schubert problem.
We construct families whose bases parameterize all instances of
$\blambda$ given by isotropic flags and whose fibers are the solutions to the
corresponding instance.

The set \defcolor{$\Fl$} of isotropic flags in $V$ is a flag manifold for
$\Sp(V)$ of dimension $m^2$.
Define
\[  
   \defcolor{\calU^*_{\blambda}} \ :=\ 
    \{ (\Fdot^1,\dotsc,\Fdot^s,\, H)\,\mid\, \Fdot^i\in\Fl\ \mbox{ and }\ 
        H\in X_{\lambda^i}\Fdot^i\ \mbox{ for }i=1,\dotsc,s\}\,.
\]
We have the two projections
\[
  \defcolor{\pi}\ \colon\ \calU^*_{\blambda}\ \longrightarrow\ (\Fl)^s
   \quad\mbox{ and }\quad
  \defcolor{\pr}\ \colon\ \calU^*_{\blambda}\ \longrightarrow\ \Gr(m,V)\,.
\]
For isotropic flags $\Fdot^1,\dotsc,\Fdot^s$, the fiber
$\pr(\pi^{-1}(\Fdot^1,\dotsc,\Fdot^s))$ consists of the solutions 
 \begin{equation}\label{Eq:fibre}
  X_{\lambda^1}\Fdot^1\ \cap\ 
  X_{\lambda^2}\Fdot^2\ \cap\ \dotsb \ \cap\ 
  X_{\lambda^s}\Fdot^s
 \end{equation}
to the instance of the Schubert problem $\blambda$ given by the flags
$\Fdot^1,\dotsc,\Fdot^s$. 

As $\Sp(V)$ does not act transitively on $\Gr(m,V)$, we cannot use 
Kleiman's Theorem~\cite{KL74} to conclude that an intersection~\eqref{Eq:fibre} given by
general flags is transverse.
Transversality follows instead from the main result of~\cite{General}.
Consequently, there is a nonempty Zariski open subset $\defcolor{\calO}\subset(\Fl)^s$
consisting of $s$-tuples of isotropic flags for which the intersection~\eqref{Eq:fibre} is
transverse and therefore consists of $d(\blambda)$ points.

We seek a family $\calX\to\calZ$ of instances of 
$\blambda$ where $\dim\calX=\dim\calZ$ and $\calZ$ is irreducible with $\calZ(\R)$
parameterizing all real instances of $\blambda$.
Since we cannot easily compute the dimension of $\calU^*_{\blambda}$, we replace
it by a possibly smaller set.
Define \defcolor{$\calU_{\blambda}$} to be the closure of $\pi^{-1}(\calO)$ in
$\calU^*_{\blambda}$.
Restricting $\pi$ to $\calU_{\blambda}$ gives the dominant map
 \begin{equation}\label{Eq:almost}
   \pi\ \colon\ \calU_{\blambda}\ \longrightarrow\ (\Fl)^s\,,
 \end{equation}
where a fiber $\pi^{-1}(\Fdot^1,\dotsc,\Fdot^s)$ is a subset of the
intersection~\eqref{Eq:fibre} and is equal to it when the intersection is finite.
Thus $\dim\calU_{\blambda}=\dim(\Fl)^s=s\cdot m^2$.

This family~\eqref{Eq:almost} has the fault that the real points
of its base $(\Fl)^s$ are $s$-tuples of real isotropic flags, which are only some of the
flags giving real instances of $\blambda$.

Let $\defcolor{S_{\blambda}}\subset S_s$ be the group of permutations $\sigma$ of
$\{1,2,\dotsc,s\}$ with $\lambda^i=\lambda^{\sigma(i)}$ for all $i=1,\dotsc,s$.
Then $S_{\blambda}\simeq S_{a_1}\times\dotsb\times S_{a_t}$ where $\blambda$ consists of
$t$ distinct partitions $\mu^1,\dotsc,\mu^t$ with $\mu^i$ occurring $a_i$ times.
Then $S_{\blambda}$ acts on the families 
$\calU^*_{\blambda},\calU_{\blambda} \to (\Fl)^s$, preserving fibers,
\[
   \pr(\pi^{-1}( \Fdot^1,\dotsc,\Fdot^s))\ =\ 
   \pr(\pi^{-1}( \Fdot^{\sigma(1)},\dotsc,\Fdot^{\sigma(s)}))
   \quad\mbox{ for all }\sigma\in S_{\blambda}\,.
\]
Define \defcolor{$\pi\colon\calX_{\blambda}\to\calZ_{\blambda}$} to be the quotient of 
$\calU_{\blambda} \to (\Fl)^s$ by the group $S_{\blambda}$.

\subsection{Proof of Theorem~\ref{Th:main}}

We defer the proof of the following lemma.

\begin{lemma}\label{L:A}
 The map $\pi\colon\calX_{\blambda}\to\calZ_{\blambda}$ is a proper dominant map of real
 varieties of the same dimension with $\calZ_{\blambda}$ smooth and $\calZ_{\blambda}(\R)$
 connected. 
 The Lagrangian involution preserves fibers of $\pi$ 
 and the codimension in $\calX_{\blambda}$ of the $\angle$-fixed points
 $\calX_{\blambda}^\angle$ is at least
 $\frac{1}{2}(\sum_i \ell(\lambda^i)-m)$.
\end{lemma}

We recall Lemma~5 from~\cite{HSZ}.

\begin{proposition}\label{P:key}
 Let $f\colon X\to Z$ be a proper dominant map of real varieties of the
 same dimension with $Z$ smooth.
 Suppose that $X$ has an involution $\angle$ preserving the fibers of $f$
 such that the image in $Z$ of the set of $\angle$-fixed points has
 codimension at least $2$. 

 If $y,z\in Z(\R)$ belong to the same connected component of\/$Z(\R)$, the fibers above 
 them are finite and at least one contains no $\angle$-fixed points, then we have 
 \[
   \# f^{-1}(y)\cap X(\R)\ \equiv\ \# f^{-1}(z)\cap X(\R)\ \mod 4\,.
 \]
\end{proposition}

\begin{remark}
 Lemma~5 in~\cite{HSZ} requires that there are no $\angle$-fixed points
 in either fiber $\pi^{-1}(y)$ or $\pi^{-1}(z)$.
 This may be relaxed to only one fiber avoiding $\angle$-fixed points, which may be seen
 using a limiting argument along the lines of the proof of Corollary~7 in~\cite{HSZ}.
\end{remark}

\begin{proof}[Proof of Theorem~$\ref{Th:main}$]
 By Lemma~\ref{L:A}, the hypotheses of Proposition~\ref{P:key} hold, as 
 the inequality $\sum_i\ell(\lambda^i)\geq m{+}4$ implies that 
 $\codim\pi(\calX_{\blambda}^\angle)\geq\codim (\calX_{\blambda}^\angle)\geq 2$.
 Let $(\Fdot^1,\dotsc,\Fdot^s)$ be isotropic flags defining a real
 instance of the Schubert problem $\blambda$ such that~\eqref{Eq:fibre} is
 finite. 
 
 Since this instance is real, for each $i=1,\dotsc,s$ if $\overline{\Fdot^i}=\Fdot^j$,
 then $\lambda^i=\lambda^j$.
 Thus there is a permutation $\sigma\in S_{\blambda}$ such that 
 $\overline{\Fdot^i}=\Fdot^{\sigma(i)}$ for $i=1,\dotsc,s$, and so 
 the image of $(\Fdot^1,\dotsc,\Fdot^s)$ in $\calZ_{\blambda}$ is a real
 point $\defcolor{y}\in\calZ_{\blambda}(\R)$.
 We complete the proof by exhibiting a point $z\in\calZ_{\blambda}(\R)$ for
 which $\pi^{-1}(z)$ consists of $d(\blambda)$ real points, none of which are fixed by
 $\angle$. 

 For distinct $t_1,\dotsc,t_s\in\R$, the intersection
 \begin{equation}\label{Eq:osculating}
  X_{\lambda^1}\Fdot(t_1)\,\cap\,
  X_{\lambda^2}\Fdot(t_2)\,\cap\,\dotsb \,\cap\,
  X_{\lambda^s}\Fdot(t_s)
 \end{equation}
 is transverse and consists of $d(\blambda)$ real points, 
 by the Mukhin-Tarasov-Varchenko Theorem~\cite{MTV}.
 The osculating flags $\Fdot(t_i)$ are real and isotropic, and we would be done if 
 there were no $\angle$-fixed points in~\eqref{Eq:osculating}.
 Equivalently, if the intersection of the corresponding Lagrangian Schubert varieties
 were empty.
 This is unknown, but expected, as it follows from Conjecture~5.1
 in~\cite{So00} which is supported by significant evidence.

 Since the intersection~\eqref{Eq:osculating} is transverse, if 
 $(\Edot^1,\dotsc,\Edot^s)\in(\Fl)^s$ are real isotropic flags that are
 sufficiently close to the osculating flags in~\eqref{Eq:osculating}, then the intersection
 \begin{equation}\label{Eq:newIntersection}
  X_{\lambda^1}\Edot^1\,\cap\,
  X_{\lambda^2}\Edot^2\,\cap\,\dotsb \,\cap\,
  X_{\lambda^s}\Edot^s
 \end{equation}
 is transverse and consists of $d(\blambda)$ real points.
 By Kleiman's Theorem~\cite{KL74} we may also assume that
 $(\Edot^1,\dotsc,\Edot^s)$ are general in that the intersection 
 \begin{equation}\label{Eq:LagrInt}
  Y_{\lambda^1}\Edot^1\,\cap\,
  Y_{\lambda^2}\Edot^2\,\cap\,\dotsb \,\cap\,
  Y_{\lambda^s}\Edot^s
 \end{equation}
 of Lagrangian Schubert varieties is either empty or has dimension
 \begin{eqnarray*}
  \binom{m+1}{2}\ -\ \sum_{i=1}^s\|\lambda^i\|  &=&
   \binom{m+1}{2}\ -\ \frac{1}{2}\sum_{i=1}^s|\lambda^i|
                 \ -\ \frac{1}{2}\sum_{i=1}^s\ell(\lambda^i)\\
    &\leq& \frac{m^2}{2}\ +\ \frac{m}{2}\ -\ \frac{m^2}{2}\ -\ \frac{m}{2}\ -\ 2
    \ =\ -2\,.
 \end{eqnarray*}
 We conclude that~\eqref{Eq:LagrInt} is empty and therefore~\eqref{Eq:newIntersection} contains no
 Lagrangian subspaces.

 If $\defcolor{z}\in\calZ_{\blambda}(\R)$ is the image of 
 $(\Edot^1,\dotsc,\Edot^s)\in(\Fl)^s$, then the fiber $\pi^{-1}(z)$ (which
 is~\eqref{Eq:newIntersection}) consists of $d(\blambda)$ real points, none of which 
 are Lagrangian.
 This completes the proof.
\end{proof}

\begin{proof}[Proof of Lemma~\ref{L:A}]
 Consider the quotient of $(\Fl)^s$ by the group $S_{\blambda}$, which is the 
 product
\[
   \calZ_{\blambda}\ =\ 
    \Sym_{a_1}(\Fl)\ \times\ 
    \Sym_{a_2}(\Fl)\ \times\ \dotsb\ \times\ 
    \Sym_{a_t}(\Fl)\,,
\]
 where \defcolor{$\Sym_a(\Fl)$} is the quotient $(\Fl)^a/S_a$ and $\blambda$ consists of
 $t$ distinct partitions $\mu^1,\dotsc,\mu^t$ with $\mu^i$ occurring $a_i$ times in
 $\blambda$. 
 
 For $\Fdot\in\Fl$, let $\defcolor{Z_e^\circ\Fdot}\subset\Fl$ be those flags in linear general
 position with respect to $\Fdot$. 
 This dense subset of $\Fl$ is a Schubert variety isomorphic to $\C^{m^2}$.
 As $\Fdot$ varies in $\Fl$, these form an affine cover of $\Fl$.
 Given a finite set $\{\Fdot^1,\dotsc,\Fdot^a\}$ of isotropic flags, there is
 an isotropic flag $\Fdot$ that is simultaneously in linear general position with each
 $\Fdot^i$, so that $\{\Fdot^1,\dotsc,\Fdot^a\}\subset Z_e^\circ\Fdot$.
 Thus $(\Fl)^a$ is covered by the $S_a$-invariant affine varieties
 $(Z_e^\circ\Fdot)^a$, each isomorphic to $(\C^{m^2})^a$.
 By descent, this implies that the quotient $\Sym_a(\Fl)=(\Fl)^a/S_a$ is well-defined and
 covered by affine varieties  $(Z_e^\circ\Fdot)^a/S_a$, each isomorphic to
 $(\C^{m^2})^a/S_a\simeq(\C^{m^2})^a$, as $\C^a/S_a\simeq\C^a$.
 It follows that $\Sym_a(\Fl)$ is a smooth irreducible variety whose real points are
 connected which implies the same for $\calZ_{\blambda}$.

 The map $\pi\colon\calU^*_{\blambda}\to(\Fl)^s$ is proper as it comes from a
 projection along a Grassmannian factor.
 Its fibers are preserved by the Lagrangian involution and 
 are equal over points in an $S_{\blambda}$-orbit.
 Both properties hold for $\pi^{-1}(\calO)\to\calO$ (as $\calO$ is $S_{\blambda}$-stable)
 and therefore for $\pi\colon\calU_{\blambda}\to(\Fl)^s$.
 We conclude that $\pi$ descends to the quotient
 $\pi\colon\calX_{\blambda}\to\calZ_{\blambda}$, where it is a proper dominant map and the
 Lagrangian involution preserves its fibers.

 Since $\dim\calU_{\blambda}=\dim(\Fl)^s=s\cdot m^2$ and $S_{\blambda}$ is a finite group,
 we conclude that $\dim\calX_{\blambda}=\dim\calZ_{\blambda}=s\cdot m^2$.

 We study the $\angle$-fixed points of $\calU^*_{\blambda}$ which form the universal family, 
\[
  \defcolor{\calL_{\blambda}}\ :=\ 
       \{ (\Fdot^1,\dotsc,\Fdot^s,\, H)\,\mid\, \Fdot^i\in\Fl\ \mbox{ and }\ 
        H\in Y_{\lambda^i}\Fdot^i\ \mbox{ for }i=1,\dotsc,s\}\,.
\]
 Consider the projection $\pr\colon\calL_{\blambda}\to\LG(V)$.
 Let $H\in\LG(V)$.
 Then
 \begin{eqnarray*}
   \pr^{-1}(H)&=& \{(\Fdot^1,\dotsc,\Fdot^s,\, H)\,\mid\, 
        H\in Y_{\lambda^i}\Fdot^i\ \mbox{ for }i=1,\dotsc,s\}\\
    &\simeq&\prod_{i=1}^s \{\Fdot\in\Fl\,\mid\, H\in Y_{\lambda^i}\Fdot\}\,.
 \end{eqnarray*}
 For $\lambda$ symmetric and $H\in\LG(V)$, define
\[
  \defcolor{Z_\lambda(H)}\ :=\ 
    \{ \Fdot\in\Fl\,\mid\, H\in Y_\lambda\Fdot\}\,.
\]
 This is a Schubert subvariety of $\Fl$ of codimension $\|\lambda\|$.
 Thus
\[
  \pr^{-1}(H)\ =\ Z_{\lambda^1}(H)\,\times\,
    Z_{\lambda^2}(H)\,\times\,\dotsb\,\times\,
    Z_{\lambda^s}(H)\,,
\]
 which has codimension
 $\sum_i\|\lambda^i\|=\frac{1}{2}\sum_i\bigl(|\lambda^i|+\ell(\lambda^i)\bigr)$ in
 $(\Fl)^s$ and is irreducible as each 
 $Z_{\lambda^i}(H)$ is a Schubert variety and is therefore irreducible.
 Thus $\pr\colon\calL_{\blambda}\to\LG(V)$
 exhibits $\calL_{\blambda}$ as a fiber bundle.
 We compute its dimension,
 \begin{eqnarray*}
  \dim\calL_{\blambda} &=& \dim\LG(V)\ +\ \dim \pr^{-1}(H) 
    \ =\ \binom{m+1}{2}\ +\ s\cdot m^2\ -\ \sum_{i=1}^s \|\lambda^i\|\\
    &=& s\cdot m^2\ -\ \frac{1}{2}\Bigl(\sum_{i=1}^s\ell(\lambda^i) - m\Bigr)\,.
 \end{eqnarray*}

 Thus $\dim \calU_{\blambda}\cap\calL_{\blambda}\leq 
  s\cdot m^2-\frac{1}{2}\bigl(\sum_i\ell(\lambda^i)-m\bigr)$.
 As $\calU_{\blambda}\cap\calL_{\blambda}$ is the set of $\angle$-fixed points of
 $\calU_{\blambda}$, $\dim\calU_{\blambda}=m^2$, and $\calX_{\blambda}$ is the quotient
 of $\calU_{\blambda}$ by the finite group $S_{\blambda}$,  the $\angle$-fixed points in
 $\calX_{\blambda}$ have codimension at least  
 $\frac{1}{2}\bigl(\sum_i\ell(\lambda^i)-m\bigr)$.
\end{proof}

\begin{remark}\label{R:last}
 If $\blambda$ is a symmetric Schubert problem,
 the quantity
\[
   \sum_{i=1}^s \|\lambda^i\|\ =\ 
   \frac{1}{2}\sum_{i=1}^s (|\lambda^i|+\ell(\lambda^i))\ =\ 
   \frac{m^2}{2}+\frac{1}{2}\sum_{i=1}^s\ell(\lambda^i)
\]
 is an integer, so $\sum_i\ell(\lambda^i)$ has the same parity as $m$.
 For generic flags $(\Edot^1,\dotsc,\Edot^s)$, the intersection~\eqref{Eq:LagrInt} of
 Lagrangian Schubert varieties is a subset of the intersection~\eqref{Eq:newIntersection}
 of Schubert varieties.
 By Kleiman's Theorem, this gives the inequality
\[
   \binom{m+1}{2}-\sum_{i=1}^s\|\lambda^i\|\ \leq\ m^2-\sum_{i=1}^s|\lambda^i|\,,
\]
 which implies that $m\leq\sum_i\ell(\lambda^i)$.
 Thus the only possibilities for $\sum_i\ell(\lambda^i)$ for which Theorem~\ref{Th:main}
 does not imply a congruence modulo four are $m$ or $m{+}2$.

 When $\sum_i\ell(\lambda^i)=m$, we have $\binom{m+1}{2}=\sum_i\|\lambda^i\|$ so that
 $\blambda$ is a Schubert problem for $\LG(V)$ with $c(\blambda)$ solutions.
 That is, for general isotropic flags $\Edot^1,\dotsc,\Edot^s$ the
 intersection~\eqref{Eq:LagrInt} is transverse and consists of $c(\blambda)$ points.
 When $c(\blambda)\neq 0$ the family $\calU_{\blambda}\to(\Fl)^s$ is reducible:
 $\calL_{\blambda}$ is one component and 
 $\overline{\calU_{\blambda}\smallsetminus\calL_{\blambda}}$ is the other.

For example, the problem 
$\raisebox{-3.5pt}{\includegraphics{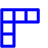}}^2\cdot\includegraphics{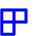}^2=8$ on
$\Gr(4,8)$ has $\sum_i\ell(\lambda^i)=4=m$.
The corresponding problem in $\LG(\C^8)$ has four solutions.
Thus four of the eight solutions on $\Gr(4,8)$ will be isotropic and the other four will not be isotropic.
In our experimentation (number 490 on~\cite{Lower_Exp}) this problem exhibits a congruence modulo four.

 When $\sum_i\ell(\lambda^i)=m+2$, a general intersection~\eqref{Eq:LagrInt} of Lagrangian
 Schubert varieties is empty and $\pi^{-1}(\calO)$ does not meet $\calL_{\blambda}$.
 There are three possibilities.
  \begin{enumerate}
   \item $\calL_{\blambda}\subset\calU_{\blambda}$ and 
         $\pi\colon \calL_{\blambda}\to\pi(\calL_{\blambda})$ generically has finite
         fibers. 
  
   \item $\calL_{\blambda}\subset\calU_{\blambda}$ and 
         $\pi\colon \calL_{\blambda}\to\pi(\calL_{\blambda})$ has positive dimensional fibers. 

   \item $\calL_{\blambda}\not\subset\calU_{\blambda}$.
  \end{enumerate}
 In case (1), $\pi(\calL_{\blambda})$ has codimension one as does the image of the set of
 $\angle$-fixed points of $\calX_{\blambda}$, so Proposition~\ref{P:key} does not
 necessarily imply a congruence modulo four. 
 In cases (2) and (3), $\pi(\calL_{\blambda})$ has codimension two, and so there will be a
 congeuence modulo four.

We have observed some symmetric Schubert problems with  $\sum_i\ell(\lambda^i)=m+2$ that have a
  congruence modulo four  and some that do not (so that (1) holds).
For example,  
$\raisebox{-3.5pt}{\includegraphics{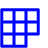}}\cdot
 \raisebox{-3.5pt}{\includegraphics{pictures/311.eps}}\cdot \includegraphics{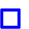}^3=6$ on $\Gr(4,8)$
   has $\sum_i\ell(\lambda^i)=6=m{+}2$  (number 495 on~\cite{Lower_Exp}) and its 
numbers of real solutions do not exhibit a congruence modulo four,
   but
$\raisebox{-3.5pt}{\includegraphics{pictures/321.eps}}^2\cdot\includegraphics{pictures/21.eps}\cdot
  \includegraphics{pictures/1.eps}=8$ on $\Gr(4,8)$ has $\sum_i\ell(\lambda^i)=6=m{+}2$ 
(number 497 on~\cite{Lower_Exp}) and its possible numbers of real solutions 
  appear to be congruent modulo four.
We believe that (3) is unlikely and that (2) holds if and only if there is a congruence modulo four.

\end{remark}

\providecommand{\bysame}{\leavevmode\hbox to3em{\hrulefill}\thinspace}
\providecommand{\MR}{\relax\ifhmode\unskip\space\fi MR }
\providecommand{\MRhref}[2]{%
  \href{http://www.ams.org/mathscinet-getitem?mr=#1}{#2}
}
\providecommand{\href}[2]{#2}


\begin{thebibliography}{10}

\bibitem{AH}
B.~Anderson and U.~Helmke, \emph{Counting critical formations on a line}, SIAM
  J. Control Optim. \textbf{52} (2014), no.~1, 219--242.

\bibitem{EH83}
D.~Eisenbud and J.~Harris, \emph{Divisors on general curves and cuspidal
  rational curves}, Invent. Math. \textbf{74} (1983), 371--418.

\bibitem{EG01}
A.~Eremenko and A.~Gabrielov, \emph{Degrees of real {W}ronski maps}, Discrete
  Comput. Geom. \textbf{28} (2002), no.~3, 331--347.

\bibitem{EG02}
\bysame, \emph{Rational functions with real critical points and the {B}. and
  {M}. {S}hapiro conjecture in real enumerative geometry}, Ann. of Math. (2)
  \textbf{155} (2002), no.~1, 105--129.

\bibitem{FM}
L\'aszl\'o~M. Feh\'er and \'Akos~K. Matszangosz, \emph{Real solutions of a
  problem in enumerative geometry}, 2014, {\tt arXiv:1401.4638}.

\bibitem{FK}
S.~Finashin and V.~Kharlamov, \emph{Abundance of real lines on real projective
  hypersurfaces}, Int. Math. Res. Notices (2012).

\bibitem{Lower_Exp}
N.~Hein and F.~Sottile, \emph{Beyond the {S}hapiro {C}onjecture and
  {E}remenko-{G}abrielov lower bounds}, 2013, \newline
  {http://www.math.tamu.edu/\~{}secant/lowerBounds/lowerBounds.php}.

\bibitem{HSZ}
N.~Hein, F.~Sottile, and I.~Zelenko, \emph{A congruence modulo four in real
  {S}chubert calculus}, 2014, Journal f\"{u}r die reine und angewandte
  {M}athematik, to appear.

\bibitem{Lower}
Nickolas Hein, Christopher Hillar, and Frank Sottile, \emph{Lower bounds in
  real {S}chubert calculus}, S\~{a}o Paulo Journal of Mathematics \textbf{7}
  (2013), no.~1, 33--58.

\bibitem{IKS03}
I.~V. Itenberg, V.~M. Kharlamov, and E.~I. Shustin, \emph{Welschinger invariant
  and enumeration of real rational curves}, Int. Math. Res. Not. (2003),
  no.~49, 2639--2653.

\bibitem{IKS04}
\bysame, \emph{Logarithmic equivalence of the {W}elschinger and the
  {G}romov-{W}itten invariants}, Uspekhi Mat. Nauk \textbf{59} (2004),
  no.~6(360), 85--110.

\bibitem{KL74}
S.L. Kleiman, \emph{The transversality of a general translate}, Compositio
  Math. \textbf{28} (1974), 287--297.

\bibitem{MSJ}
Abraham Mart\'in~del Campo and Frank Sottile, \emph{Experimentation in the
  {S}chubert calculus}, {\tt arXiv.org/1308.3284}, 2013.

\bibitem{MT}
E.~Mukhin and V.~Tarasov, \emph{Lower bounds for numbers of real solutions in
  problems of Schubert cal\-cu\-lus}, 2014, {\tt arXiv:1404.7194}.

\bibitem{MTV}
E.~Mukhin, V.~Tarasov, and A.~Varchenko, \emph{The {B}. and {M}. {S}hapiro
  conjecture in real algebraic geometry and the {B}ethe ansatz}, Ann. of Math.
  (2) \textbf{170} (2009), no.~2, 863--881.

\bibitem{MTV_R}
\bysame, \emph{Schubert calculus and representations of the general linear
  group}, J. Amer. Math. Soc. \textbf{22} (2009), no.~4, 909--940.

\bibitem{OT1}
Ch. Okonek and A.~Teleman, \emph{Intrinsic signs and lower bounds in real
  algebraic geometry}, {Journal f\"ur Reine und Angewandte Mathematik}
  \textbf{688} (2014), 219--241.

\bibitem{OT2}
\bysame, \emph{A wall-crossing formula for degrees of real central
  projections}, Int. J. Math. \textbf{25} (2014), 34 pages.

\bibitem{Purbhoo}
K.~Purbhoo, \emph{Reality and transversality for {S}chubert calculus in {${\rm
  OG}(n,2n+1)$}}, Math. Res. Lett. \textbf{17} (2010), no.~6, 1041--1046.

\bibitem{SS}
E.~Soprunova and F.~Sottile, \emph{Lower bounds for real solutions to sparse
  polynomial systems}, Adv. Math. \textbf{204} (2006), no.~1, 116--151.

\bibitem{So99}
F.~Sottile, \emph{The special {S}chubert calculus is real}, Electronic Research
  Anouncements of the AMS \textbf{5} (1999), 35--39.

\bibitem{So00}
\bysame, \emph{Some real and unreal enumerative geometry for flag manifolds},
  Mich. Math. J. \textbf{48} (2000), 573--592, Special Issue in Honor of
  Wm.~Fulton.

\bibitem{General}
\bysame, \emph{General isotropic flags are general (for {G}rassmannian
  {S}chubert calculus)}, J. Algebraic Geom. \textbf{19} (2010), no.~2,
  367--370.

\bibitem{W}
J.-Y. Welschinger, \emph{Invariants of real rational symplectic 4-manifolds and
  lower bounds in real enumerative geometry}, C. R. Math. Acad. Sci. Paris
  \textbf{336} (2003), no.~4, 341--344.

\end{thebibliography}
\end{document}